\documentclass[leqno,12pt]{amsart}
\usepackage{amssymb}
\usepackage{amsmath}
\usepackage{enumerate}
\usepackage{amsfonts}
\usepackage{hyperref}
\usepackage{tikz}

\newcommand{\supp}{\text{\rm supp}\,}
\newcommand{\diam}{\text{\rm diam}}
\newtheorem{teo}[equation]{Theorem}

\newtheorem{coro}[equation]{Corollary}
\newtheorem{lema}[equation]{Lemma}
\newtheorem{propo}[equation]{Proposition}
\newtheorem{remark}[equation]{Remark}
\newtheorem{defn}[equation]{Definition}

\numberwithin{equation}{section}

\hypersetup{
    colorlinks=true,
    linkcolor= blue,
    citecolor =cyan,
    urlcolor = teal,
}

\begin{document}

\title[H\"ormander's multiplier theorem for the Dunkl transform]{H\"ormander's multiplier theorem\\ for the Dunkl transform}

\author[ J. Dziuba\'nski and A. Hejna]{Jacek Dziuba\'nski and Agnieszka Hejna}


\begin{abstract} For a normalized root system $R$ in $\mathbb R^N$ and a multiplicity function $k\geq 0$ let $\mathbf N=N+\sum_{\alpha \in R} k(\alpha)$. Denote by $dw(\mathbf x)=\prod_{\alpha\in R}|\langle \mathbf x,\alpha\rangle|^{k(\alpha)}\, d\mathbf x $ the associated measure in $\mathbb R^N$.  Let $\mathcal F$ stands for the Dunkl transform. Given a bounded function $m$ on $\mathbb R^N$,  
we prove that if  there is $s>\mathbf N$ such that $m$ satisfies the classical H\"ormander condition with the smoothness  $s$, then
the multiplier operator $\mathcal T_mf=\mathcal F^{-1}(m\mathcal Ff)$ is of weak type $(1,1)$,  strong type $(p,p)$ for $1<p<\infty$, and bounded on a relevant Hardy space $H^1$. To this end we study the Dunkl translations and the Dunkl convolution operators and prove that if $F$ is sufficiently regular, for example its certain Schwartz class seminorm is finite, then the Dunkl convolution operator with the function $F$ is  bounded  on $L^p(dw)$ for $1\leq p\leq \infty$. We also consider boundedness of maximal operators associated with the Dunkl convolutions with Schwartz class functions. 
\end{abstract}

\address{J. Dziuba\'nski and A. Hejna, Uniwersytet Wroc\l awski,
Instytut Matematyczny,
Pl. Grunwaldzki 2/4,
50-384 Wroc\l aw,
Poland}
\email{jdziuban@math.uni.wroc.pl}
\email{hejna@math.uni.wroc.pl}

\thanks{
Research supported by the National Science Centre, Poland (Narodowe Centrum Nauki), Grant 2017/25/B/ST1/00599.}

\maketitle

\section{Introduction and statements of the results}

On the Euclidean space $\mathbb R^N$ we consider  a normalized root system $R$ and a  multiplicity function $k\geq 0$.  Let
$$dw(\mathbf x)=\prod_{\alpha\in R}|\langle \mathbf x,\alpha\rangle|^{k(\alpha)}\, d\mathbf x$$ be  the associated measure in $\mathbb R^N$, where, here and subsequently, $d\mathbf x$ stands for the Lebesgue measure in $\mathbb R^N$.
 Denote by $\mathbf N=N+\sum_{\alpha \in R} k(\alpha)$ the homogeneous dimension of the system and by $G$ the Weyl group generated by the reflections $\sigma_{\alpha}$, $\alpha\in R$. Let $E(\mathbf x,\mathbf y)$ be the associated Dunkl kernel. The kernel $E(\mathbf x,\mathbf y)$ has a unique extension to a holomorphic function in $\mathbb C^N\times\mathbb C^N$. The Dunkl transform
  \begin{equation}\label{eq:DunklTransform}\mathcal F f(\xi)=c_k^{-1}\int_{\mathbb R^N} E(-i\xi, \mathbf x)f(\mathbf x)\, dw(\mathbf x),
  \end{equation}
  where
  $$c_k=\int_{\mathbb{R}^N}e^{-\frac{\|\mathbf{x}\|^2}{2}}\,dw(\mathbf{x})>0,$$
   originally defined for $f\in L^1(dw)$, is an isometry on $L^2(dw)$ and preserves the Schwartz class of functions $\mathcal S(\mathbb R^N)$ (see \cite{deJeu}). Its inverse $\mathcal F^{-1}$ has the form
  $$\mathcal F^{-1} g(x)=c_k^{-1}\int_{\mathbb R^N} E(i\xi, \mathbf x)g(\xi)\, dw(\xi).$$
The Dunkl transform $\mathcal F$ is an analogue of the classical Fourier transform
  $$ \widehat f(\xi)=\int_{\mathbb R^N} e^{-i\langle \xi,\mathbf x\rangle} f(\mathbf x)\, d\mathbf x.$$

Let 
$$\| m\|_{W_2^{s}}=\| \widehat m(\mathbf x)(1+\| \mathbf x\|)^s\|_{L^2(d\mathbf x)}$$for $s \geq 0$ be the classical Sobolev norm.

The  metric measure space $(\mathbb R^N, \| \mathbf x -\mathbf y\|, dw)$ is doubling (see \eqref{eq:doubling}). Let $H^1_{\text{\rm atom}} $ denote the atomic Hardy space in the sense of Coifman--Weiss \cite{CW} on the space of homogeneous type  $(\mathbb R^N, \| \mathbf x -\mathbf y\|, dw)$ (see Section~\ref{sec:Hardy} for details). 

We are in a position to state our main result.

\begin{teo}\label{teo:main_teo}
Let $\psi$ be a smooth radial function such that $\supp \psi \subseteq \{\xi:\frac{1}{4} \leq \|\xi\| \leq 4\}$ and $\psi (\xi) \equiv 1$ for $\{\xi:\frac{1}{2} \leq \|\xi\| \leq 2\}$. If $m$ is a  function on $\mathbb R^N$ which satisfies the H\"ormander condition
\begin{equation}\label{eq:assumption}
M=\sup_{t>0}\|\psi (\cdot)m(t \cdot )\|_{W^{s}_2} < \infty
\end{equation}
for some $s>\mathbf N$, then the multiplier operator
\begin{align*}
\mathcal{T}_mf=\mathcal{F}^{-1}(m\mathcal{F}f),
\end{align*}
originally defined on $L^2(dw)\cap L^1(dw)$, is of
\begin{enumerate}[(A)]
\item{weak type $(1,1)$,}\label{numitem:weak_type}
\item{strong type $(p,p)$ for $1<p<\infty$,}\label{numitem:strong_type}
\item{bounded on the Hardy space $H^1_{\rm atom}$.}\label{numitem:hardy}
\end{enumerate}
\end{teo}

Let us remark that we need the regularity of order $s>\mathbf N$ in the H\"ormander's condition \eqref{eq:assumption}. One might  expect that a regularity $s>\mathbf N\slash 2$ would suffice (see the classical H\"ormander's multiplier theorem for the Fourier transform ~\cite{Hormander}). The price of $\mathbf N\slash 2$ we pay in the proof of the theorem is due to the fact that the so called Dunkl translation
\begin{equation}\label{eq:translation} \tau_{\mathbf x}f(\mathbf y)=c_k^{-1}\int E(i\xi , \mathbf x)E(i\xi , \mathbf y)\mathcal Ff(\xi)\, dw(\xi)
\end{equation}
is  bounded on $L^2(dw)$ and it is an open problem if it is bounded on $L^p(dw)$ for $p\ne 2$.
 If the translations $\tau_{\mathbf x}$ for a system of roots are uniformly  bounded operators  on $L^1(dw)$, then the the regularity of order $s>\mathbf N\slash 2$  suffices.  We elaborate this situation in Section \ref{sec:boundedL1}. This happens e.g. in  the case of the product system of roots or radial multipliers (see \cite{ABDH} and \cite{DaiXu}).
So in order to overcome the lack of knowledge about the Dunkl translation on the $L^p(dw)$-spaces for a general system of roots  we use the only information we have, that is, the  boundedness of $\tau_{\mathbf x} $ on $L^2(dw)$ together with very important observation about supports of translations of $L^2(dw)$-functions, which is stated in the following our next main result.

\begin{teo}\label{teo:support}
Let $f \in L^2(dw)$, $\text{\rm supp}\, f \subseteq B(0,r)$, and $\mathbf{x} \in \mathbb{R}^N$. Then  
$$\supp \tau_{\mathbf{x}}f(-\, \cdot) \subseteq \mathcal{O}(B(\mathbf x,r)),$$
where $\mathcal O(B(\mathbf x, r))=\bigcup_{\sigma \in G} B(\sigma (\mathbf x), r))$ is the orbit of the Euclidean closed ball $B(\mathbf x,r)=\{\mathbf y \in \mathbb{R}^N: \| \mathbf x-\mathbf y\|\leq r\}$. 
\end{teo}

The conclusion of Theorem~\ref{teo:support} is known for $f$ being $L^2(dw)$-radial functions supported by $B(0,r)$ (see~\eqref{eq:support_radial}). Our aim is to extend it for functions  which are not necessary radial.  

Let us also note that the Theorem \ref{teo:support} gives much precise information about the support of translations than that which follows from~\cite[Theorem 5.1]{AAS-Colloq}. Actually their result implies that $\supp\, \tau_{\mathbf{x}}f(-\, \cdot) \subset \{\mathbf y\in \mathbb R^N: \| \mathbf x\|-r\leq \| \mathbf y\|\leq \| \mathbf x\|+r\}$ for $f\in L^2(dw)$, $\text{supp}\, f\subseteq B(0,r)$.

 \section{Preliminaries}
Dunkl theory is a generalization of Euclidean Fourier analysis. It started with the seminal article \cite{Dunkl} and developed extensively afterwards (see e.g.,\cite{RoeslerDeJeu}, \cite{Dunkl0}, \cite{Dunkl2}, \cite{Dunkl3}, \cite{GR}, \cite{Roesler2}, \cite{Roesler2003}, \cite{Roesle99}, \cite{ThangaveluXu}, \cite{Trimeche2002}). 
In this section we present basic facts concerning theory of the Dunkl operators.  For details we refer the reader to~\cite{Dunkl},~\cite{Roesler3}, and~\cite{Roesler-Voit}.

We consider the Euclidean space $\mathbb R^N$ with the scalar product $\langle\mathbf x,\mathbf y\rangle=\sum_{j=1}^N x_jy_j
$, $\mathbf x=(x_1,...,x_N)$, $\mathbf y=(y_1,...,y_N)$, and the norm $\| \mathbf x\|^2=\langle \mathbf x,\mathbf x\rangle$. For a nonzero vector $\alpha\in\mathbb R^N$ the reflection $\sigma_\alpha$ with respect to the hyperplane $\alpha^\perp$ orthogonal to $\alpha$ is given by
\begin{equation}\label{eq:reflection}
\sigma_\alpha\mathbf x=\mathbf x-2\frac{\langle \mathbf x,\alpha\rangle}{\| \alpha\| ^2}\alpha.
\end{equation}
A finite set $R\subset \mathbb R^N\setminus\{0\}$ is called a {\it root system} if $\sigma_\alpha (R)=R$ for every $\alpha\in R$. We shall consider normalized reduced root systems, that is, $\|\alpha\|^2=2$  for every $\alpha\in R$. The finite group $G$ generated by the reflections $\sigma_\alpha$ is called the {\it Weyl group} ({\it reflection group}) of the root system. A {\textit{multiplicity function}} is a $G$-invariant function $k:R\to\mathbb C$ which will be fixed and $\geq 0$  throughout this paper.

The number $\mathbf N$ is called the homogeneous dimension of the system, since
\begin{equation}\label{eq:homo} w(B(t\mathbf x, tr))=t^{\mathbf N}w(B(\mathbf x,r)) \ \ \text{\rm for } \mathbf x\in\mathbb R^N, \ t,r>0,
\end{equation}
where here and subsequently $B(\mathbf x, r)=\{\mathbf y\in\mathbb R^N:\|\mathbf x-\mathbf y\|\leq r\}$ denotes the (closed) Euclidean ball centered at $\mathbf x$ with radius $r>0$. Observe that 
\begin{equation}\label{eq:behavior} w(B(\mathbf x,r))\sim r^{N}\prod_{\alpha \in R} (|\langle \mathbf x,\alpha\rangle |+r)^{k(\alpha)},
\end{equation}
so $dw(\mathbf x)$ it is doubling, that is, there is a constant $C>0$ such that
\begin{equation}\label{eq:doubling} w(B(\mathbf x,2r))\leq C w(B(\mathbf x,r)) \ \ \text{\rm for } \mathbf x\in\mathbb R^N, \ r>0.
\end{equation}
Moreover, by \eqref{eq:behavior},
\begin{equation}\label{eq:growth}
C^{-1} \left(\frac{R}{r}\right)^{ N}\leq \frac{w(B(\mathbf x, R))}{w(B(\mathbf x, r))}\leq C \left(\frac{R}{r}\right)^{\mathbf N}\ \ \text{ for } 0<r<R.
\end{equation}

The {\it Dunkl operator} $T_\xi$  is the following $k$-deformation of the directional derivative $\partial_\xi$ by a  difference operator:
$$ T_\xi f(\mathbf x)= \partial_\xi f(\mathbf x) + \sum_{\alpha\in R} \frac{k(\alpha)}{2}\langle\alpha ,\xi\rangle\frac{f(\mathbf x)-f(\sigma_\alpha \mathbf x)}{\langle \alpha,\mathbf x\rangle}.$$
The Dunkl operators $T_{\xi}$, which were introduced in~\cite{Dunkl}, commute and are skew-symmetric with respect to the $G$-invariant measure $dw$.

Let $e_j$, $j=1,2,...,N$, denote the canonical orthonormal basis in $\mathbb R^N$ and let $T_j=T_{e_j}$.

For fixed $\mathbf y\in\mathbb R^N$ the {\it Dunkl kernel} $E(\mathbf x,\mathbf y)$ is the unique solution of the system
$$ T_\xi f=\langle \xi,\mathbf y\rangle f, \ \ f(0)=1.$$
In particular
\begin{equation}\label{eq:T_j_x} T_{j,\mathbf x} E(\mathbf x,\mathbf y)=y_jE(\mathbf x,\mathbf y),
\end{equation}
where here and subsequently  $T_{j,\mathbf x}$  denotes the action of $T_j$ with respect to the variable $\mathbf x$.
The function $E(\mathbf x ,\mathbf y)$, which generalizes the exponential  function
 $e^{\langle \mathbf x,\mathbf y\rangle}$, has a unique extension to a holomorphic function on $\mathbb C^N\times \mathbb C^N$. We have (see, e.g. \cite{Roesler3}, \cite{Roesler-Voit})
 \begin{itemize}
\item[$\bullet$]
$E(\lambda \mathbf x,\mathbf y)=E(\mathbf x,\lambda \mathbf y)=E(\lambda  \mathbf y,\mathbf x)=E(\lambda \sigma(\mathbf x),\sigma(\mathbf y))$ for all  $\mathbf x,\mathbf y\in \mathbb C^N$, $\sigma \in G$, and $\lambda \in \mathbb C$;
\item[$\bullet$]
$E(\mathbf x,\mathbf y)>0$ for all $\mathbf x,\mathbf y\in\mathbb R^N$;
\item[$\bullet$]
$|E(-i\mathbf x,\mathbf y)|\leq 1 $ for all $\mathbf x,\mathbf y\in\mathbb R^N$;
\item[$\bullet$]
$E(0,\mathbf y)=1$ for all $\mathbf y\in \mathbb C^N$.
\end{itemize}
Let us collect basic properties of the Dunkl transform $\mathcal F$ and the Dunkl translation ${\tau}_{\mathbf x}$ defined in \eqref{eq:DunklTransform} and  \eqref{eq:translation}

\begin{itemize}
\item[$\bullet$]
 $ \mathcal F(T_\zeta f)=i\langle \zeta , \cdot\rangle \mathcal Ff, \ \ \text{\rm and} \ T_\zeta (\mathcal Ff)=-i\mathcal F(\langle \zeta,\cdot\rangle f)$;

\item[$\bullet$]
 the Dunkl transform of a radial function is again a radial function; 

\item[$\bullet$]
 $ \mathcal F(f_\lambda)(\xi)=\mathcal F(\lambda \xi)$, where $f_\lambda (\mathbf x)=\lambda^{-\mathbf N} f(\lambda^{-1} \mathbf x)$, $\lambda>0$; 

\item[$\bullet$]
each translation $\tau_{\mathbf{x}}$ is a continuous linear map of $\mathcal{S}(\mathbb{R}^N)$ into itself, which extends to a contraction on $L^2({dw})$;
\item[$\bullet$]
(\textit{Identity\/})
$\tau_0=I$;
\item[$\bullet$]
(\textit{Symmetry\/})
$\tau_{\mathbf{x}}f(\mathbf{y})
=\tau_{\mathbf{y}}f(\mathbf{x})\text{ for all }\mathbf{x},\mathbf{y}\in\mathbb{R}^N, f\in\mathcal{S}(\mathbb{R}^N)$;
\item[$\bullet$]
(\textit{Scaling\/})
$\tau_{\mathbf{x}}(f_\lambda)=(\tau_{\lambda^{-1}\mathbf{x}}f)_\lambda \text{ for all }\lambda>0\,,\mathbf{x}
\in\mathbb{R}^N,\,f\in\mathcal{S}(\mathbb{R}^N)$;
\item[$\bullet$]
(\textit{Commutativity\/}) $T_\xi(\tau_{\mathbf x} f)=\tau_{\mathbf x} (T_\xi f)$;
\item[$\bullet$]
(\textit{Skew--symmetry\/})
\begin{equation*}\quad
\int_{\mathbb{R}^N}\!\tau_{\mathbf{x}}f(\mathbf{y})\,g(\mathbf{y})\,dw(\mathbf{y})=\int_{\mathbb{R}^N}f(\mathbf{y})\,\tau_{-\mathbf{x}}g(\mathbf{y})\,dw(\mathbf{y})
\text{ for all }\mathbf{x}\in\mathbb{R}^N,\,f,g\in\mathcal{S}(\mathbb{R}^N).
\end{equation*}
\end{itemize}

The latter formula allows us to define the Dunkl translations $\tau_{\mathbf{x}}f$ in the distributional sense for $f\in L^p({dw})$ with $1\leq p\leq \infty$. Further,
\begin{equation*}\quad
\int_{\mathbb{R}^N}\!\tau_{\mathbf{x}}f(\mathbf{y})\,dw(\mathbf{y})=\int_{\mathbb{R}^N}f(\mathbf{y})\,dw(\mathbf{y})
\text{ for all }\mathbf{x}\in\mathbb{R}^N,\,f\in\mathcal{S}(\mathbb{R}^N).
\end{equation*}

The following specific formula was obtained by R\"osler \cite{Roesler2003}: for the Dunkl translations of (reasonable) radial functions $f({\mathbf{x}})=\tilde{f}({\|\mathbf{x}\|})$:
\begin{equation}\label{eq:translation-radial}
\tau_{\mathbf{x}}f(-\mathbf{y})=\int_{\mathbb{R}^N}{(\tilde{f}\circ A)}(\mathbf{x},\mathbf{y},\eta)\,d\mu_{\mathbf{x}}(\eta)\text{ for all }\mathbf{x},\mathbf{y}\in\mathbb{R}^N.
\end{equation}
Here
\begin{equation*}
A(\mathbf{x},\mathbf{y},\eta)=\sqrt{{\|}\mathbf{x}{\|}^2+{\|}\mathbf{y}{\|}^2-2\langle \mathbf{y},\eta\rangle}=\sqrt{{\|}\mathbf{x}{\|}^2-{\|}\eta{\|}^2+{\|}\mathbf{y}-\eta{\|}^2}
\end{equation*}
and $\mu_{\mathbf{x}}$ is a probability measure, which is supported in $\operatorname{conv}\mathcal{O}(\mathbf{x})$. 

It is not hard to see that  $A(\mathbf x,\mathbf y,\eta)\geq d(\mathbf x,\mathbf y)$ for $\mathbf x,\mathbf y\in \mathbb R^N$ and $\eta \in  \operatorname{conv}\mathcal{O}(\mathbf{x})$, where here and subsequently, 
 \begin{equation}\label{eq:orbit_dist} d(\mathbf x,\mathbf y)=\inf_{\sigma \in G}
\| \sigma(\mathbf x)-\mathbf y\|
\end{equation}
denotes the distance of the orbits $\mathcal O(\mathbf x)$ and $\mathcal O(\mathbf y)$. Hence,~\eqref{eq:translation-radial} implies that if $f\in L^2(dw)$ is a radial function supported by $B(0,r)$, then 
\begin{equation}\label{eq:support_radial} 
 \supp \tau_{\mathbf{x}}f(-\, \cdot) \subseteq \mathcal{O} (B(\mathbf x,r)).  
 \end{equation}
We prove first \eqref{eq:support_radial} for radial  $C_c^\infty$-functions and then use a density argument and  continuity of the Dunkl translation on $L^2(dw)$.

{The \textit{Dunkl convolution\/} of two reasonable functions (for instance Schwartz functions) is defined by
$$
(f*g)(\mathbf{x})=c_k\,\mathcal{F}^{-1}[(\mathcal{F}f)(\mathcal{F}g)](\mathbf{x})=\int_{\mathbb{R}^N}(\mathcal{F}f)(\xi)\,(\mathcal{F}g)(\xi)\,E(\mathbf{x},i\xi)\,dw(\xi)\text{ for all }\mathbf{x}\in\mathbb{R}^N,
$$
or, equivalently, by}
$$
{(}f*g{)}(\mathbf{x})
=\int_{\mathbb{R}^N}f(\mathbf{y})\,\tau_{\mathbf{x}}g(-\mathbf{y})\,{dw}(\mathbf{y})
=\int f(\mathbf y) g(\mathbf x,\mathbf y)\, dw(\mathbf y),$$
where here and subsequently, we use the notation 
 \begin{equation}\label{eq:notation} g(\mathbf x,\mathbf y)=\tau_{\mathbf x}g(-\mathbf y)=\tau_{-\mathbf y}g(\mathbf x)
 \end{equation} for a reasonable function $g(\mathbf x)$ on $\mathbb R^N$. The last equality in~\eqref{eq:notation} follows by symmetry of the Dunkl translation. 

Let us collect some well-known formulae for the Dunkl translation and convolution.
\begin{equation}\label{eq:t1}
\|\tau_{\mathbf y} f\|_{L^2(dw)}\leq \| f\|_{L^2(dw)} \ \ \text{for} \ f\in L^2(dw),
\end{equation}
\begin{equation}\label{eq:t2}
\| f*g\|_{L^2(dw)}\leq \| f\|_{L^1(dw)}\| g\|_{L^2(dw)} \ \ \text{for}  \ f\in L^1(dw), \ g\in L^2(dw).
\end{equation}

If $g$ is a radial and continuous compactly supported, then
$\| \tau_{\mathbf y} g\|_{L^1(dw)}=\| g\|_{L^1(dw)}$ (see~\eqref{eq:translation-radial}), in particular the translation $\tau_{\mathbf y}g$ can be uniquely extended to radial $L^1(dw)$-functions. Hence one can give a sense for the convolution $f*g$, where $f\in L^p(dw)$ and $g\in L^1(dw)$ is radial.

The {\it Dunkl Laplacian} associated with $G$ and $k$  is the differential-difference operator $\Delta=\sum_{j=1}^N T_{j}^2$, which  acts on $C^2(\mathbb{R}^N)$ functions by
 $$ \Delta f(\mathbf x)=\Delta_{\rm eucl} f(\mathbf x)+\sum_{\alpha\in R} k(\alpha) \delta_\alpha f(\mathbf x),$$
$$\delta_\alpha f(\mathbf x)=\frac{\partial_\alpha f(\mathbf x)}{\langle \alpha , \mathbf x\rangle} - \frac{\|\alpha\|^2}{2} \frac{f(\mathbf x)-f(\sigma_\alpha \mathbf x)}{\langle \alpha, \mathbf x\rangle^2}.$$
Clearly, $\mathcal F(\Delta f)(\xi)=-\| \xi\|^2\mathcal Ff(\xi)$. The operator $\Delta$ is essentially self-adjoint on $L^2(dw)$  (see for instance \cite[Theorem\;3.1]{AH}) and generates the semigroup $e^{t\Delta}$  of linear self-adjoint contractions on $L^2(dw)$. The semigroup has the form
  \begin{equation}\label{eq:heat_semigroup}
  e^{t\Delta} f(\mathbf x)=\mathcal F^{-1}(e^{-t\|\xi\|^2}\mathcal Ff(\xi))(\mathbf x)=\int_{\mathbb R^N} h_t(\mathbf x,\mathbf y)f(\mathbf y)\, dw(\mathbf y),
  \end{equation}
  where the heat kernel $h_t(\mathbf x,\mathbf y)=\tau_{\mathbf x}h_t(-\mathbf y)$, $h_t(\mathbf x)=\mathcal F^{-1} (e^{-t\|\xi\|^2})(\mathbf x)=c_k^{-1} (2t)^{-\mathbf N\slash 2}e^{-\| \mathbf x\|^2\slash (4t)}$ is $C^\infty$ function of all variables $\mathbf x,\mathbf y \in \mathbb{R}^N$, $t>0$ and satisfies \begin{equation}\label{eq:symmetry} 0<h_t(\mathbf x,\mathbf y)=h_t(\mathbf y,\mathbf x),
  \end{equation}
 \begin{equation}\label{eq:heat_one} \int_{\mathbb R^N} h_t(\mathbf x,\mathbf y)\, dw(\mathbf y)=1.
 \end{equation}
 
\section{Properties of Dunkl translations - proof of Theorem \ref{teo:support}}
\subsection{Properties of translations of the Dunkl heat kernel}
We start this section by the list of further properties of the heat kernel.  
 Set
$$V(\mathbf x,\mathbf y,t)=\max (w(B(\mathbf x,t)),w(B(\mathbf y, t))).$$ 
The following estimates were proved in~\cite[Theorem 4.3]{conjugate}.
\begin{teo}
There are constants $C,c>0$ such that for all $\mathbf{x},\mathbf{y} \in \mathbb{R}^N$ and $t>0$ we have
\begin{equation}\label{eq:Gauss}
\left|h_t(\mathbf{x},\mathbf{y})\right|\leq C\,V(\mathbf{x},\mathbf{y},\!\sqrt{t\,})^{-1}\,e^{-\hspace{.25mm}c\hspace{.5mm}d(\mathbf{x},\mathbf{y})^2\slash t},
\end{equation}
\begin{equation}\label{eq:Holder}
\left|h_t(\mathbf{x},\mathbf{y})-h_t(\mathbf{x},\mathbf{y}')\right|\leq C\,\Bigl(\frac{{\|}\mathbf{y}\!-\!\mathbf{y}'{\|}}{\sqrt{t\,}}\Bigr)\,V(\mathbf{x},\mathbf{y},\!\sqrt{t\,})^{-1}\,e^{-\hspace{.25mm}c\hspace{.5mm}d(\mathbf{x},\mathbf{y})^2\slash t}.
\end{equation}
\end{teo}
Clearly, $\tau_{\mathbf x}\tau_{\mathbf y} h_t(-\mathbf z)$ is $C^\infty$-function of $\mathbf x,\mathbf y,\mathbf z\in \mathbb R^N$ and $t>0$. The following formula is a direct consequence of \eqref{eq:Gauss}.
\begin{equation}\label{eq:hL^2}
\| h_t(\mathbf x,\mathbf y)\|_{L^2(dw(\mathbf y))}\leq \frac{C}{w(B(\mathbf x, \sqrt{t}))^{1\slash 2}}.
\end{equation}
We have 
\begin{equation}\label{eq:h_infty}
|\tau_{\mathbf x}\tau_{\mathbf y}  h_t(-\mathbf z)|\leq C\ w(B(\mathbf x, \sqrt{t}))^{-1\slash 2}w(B(\mathbf y,\sqrt{t}))^{-1\slash 2},
\end{equation}
\begin{equation}\label{eq:IntegralE}
\int_{B(0, 1\slash t)} |E(i\xi,\mathbf x)|^2\, dw(\xi)\leq \frac{C}{w(B(\mathbf x,t))}.
\end{equation}

\begin{proof}[Proof of~\eqref{eq:h_infty}]
By Cauchy--Schwarz inequality, Plancharel theorem for the Dunkl transform $\mathcal{F}$, and~\eqref{eq:hL^2} we have
\begin{align*}
|\tau_{\mathbf{x}}\tau_{\mathbf{y}}h_{t}(-\mathbf{z})|&=c_k^{-1}\left|\int E(i\xi,\mathbf{x})E(i\xi,\mathbf{y})E(-i\xi,\mathbf{z})e^{-t\|\xi\|^2}\,dw(\xi)\right| \\&\leq c_{k}^{-1}\left(\int |E(i\xi,\mathbf{x})|^2e^{-t\|\xi\|^{2}}\,dw(\xi)\right)^{1/2}\left(\int |E(i\xi,\mathbf{y})|^2e^{-t\|\xi\|^{2}}\,dw(\xi)\right)^{1/2}\\&=\| h_t(\mathbf x,\mathbf z)\|_{L^2(dw(\mathbf z))}\| h_t(\mathbf y,\mathbf z)\|_{L^2(dw(\mathbf z))}\leq C\ w(B(\mathbf x, \sqrt{t}))^{-1\slash 2}w(B(\mathbf y,\sqrt{t}))^{-1\slash 2}.
\end{align*}
\end{proof}
\begin{proof}[Proof of~\eqref{eq:IntegralE}]
By~\eqref{eq:hL^2} and Plancherel theorem for the Dunkl transform $\mathcal F$ we get  
\begin{align*}
\frac{C}{w(B(\mathbf{x},t))^{1/2}} \geq \|h_{t^2}(\mathbf{x},\cdot)\|_{L^2(dw)}
=\|E(i\cdot,\mathbf{x})e^{-t^2\|\cdot\|^2}\|_{L^2(dw)}
\geq e^{-1}\|E(i\cdot,\mathbf{x})\|_{L^2(B(0,t^{-1}),dw)}.
\end{align*}\end{proof}

\begin{propo} If $m$ is a bounded function supported by $B(0,1\slash t)$, then $(\mathcal{F}^{-1}m)(\mathbf x,\mathbf y)$ is a $C^\infty$-function of $\mathbf x,\mathbf y\in\mathbb R^N$ which satisfies
\begin{equation}\label{eq:double_ball}
| (\mathcal{F}^{-1}m)(\mathbf x,\mathbf y)|\leq C \| m\|_{L^\infty} w(B(\mathbf x,t))^{-1\slash 2} w(B(\mathbf y, t))^{-1\slash 2}.
\end{equation}
\end{propo}
\begin{proof}
By the Cauchy--Schwarz inequality and~\eqref{eq:IntegralE},
\begin{align*}
\begin{split}
\Big|(\mathcal{F}^{-1}m)(\mathbf x,\mathbf y)\Big|
&=c_k^{-1}\Big|\int_{B(0,1/t)} m(\xi)E(-i\xi,\mathbf{x})E(i\xi,\mathbf{y})\,dw(\xi)\Big|\\
&\leq c_k^{-1} \|m\|_{L^{\infty}}\|E(i\cdot,\mathbf{x})\|_{L^2(B(0,1/t),dw)}\|E(i\cdot,\mathbf{y})\|_{L^2(B(0,1/t),dw)}\\&\leq C\|m\|_{L^{\infty}}w(B(\mathbf{x},t))^{-1/2}w(B(\mathbf{y},t))^{-1/2}.
\end{split}
\end{align*}
\end{proof}

\subsection{Support of translations of compactly supported functions}
 Suppose that $f,g \in C^1(\mathbb{R}^N)$ and $g$ is radial. The following Leibniz rule can be confirmed by a direct calculation:
\begin{equation}\label{eq:Leibniz}
T_{j}(f g)=f(T_j g)+g(T_j f) \text{ for all }1 \leq j \leq N.
\end{equation}
Let us denote the set of all polynomials of degree $d \geq 0$ by $\mathbb{P}_d$. 
\begin{propo}[{\cite[Lemma 2.6]{Roesler-Voit}}]\label{propo:polynomial_degree}
Let $p \in \mathbb{P}_d$ and $1 \leq j \leq N$. Then $T_jp \in \mathbb{P}_{d-1}$.
\end{propo}

Let $L$ be a positive integer. Let us denote
\begin{align*}
g_{L}(\mathbf{x})=\max\{0,(1-\|\mathbf{x}\|^2)\}^{L}.
\end{align*}
The function $g_{L}$ is radial, belongs to $C^{L-1}(\mathbb{R}^N)$, and $\supp g_L\subseteq B(0,1)$. 

For $\mathbf{\alpha}=(\alpha_1,\alpha_2,\ldots,\alpha_N) \in \mathbb N_0^N=(\mathbb{N} \cup \{0\})^{N} $ we define
\begin{align*}
T_j^{0}=I, \ \ T^{\alpha}:=T_{1}^{\alpha_1} \circ T_2^{\alpha_2} \circ \ldots \circ T_{N}^{\alpha_N}.
\end{align*}
Clearly, $\text{supp}\, T^\alpha g_L\subseteq B(0,1)$ for $\| \alpha\|<L$, where $\| \alpha \|=\sum_{j=1}^N \alpha_j$. 
\begin{lema}\label{lem:representation}
Let $L \in \mathbb{N}$ and $p$ be a polynomial of degree $d$. Then $pg_{L}$ can be written in the form
\begin{equation}\label{eg:polynomial_form}
p(\mathbf{x})g_{L}(\mathbf{x})=\sum_{\ell=0}^{d}\sum_{\|\alpha\| \leq \ell}c_{\ell,\alpha}T^{\alpha}(g_{L+\ell})(\mathbf{x})
\end{equation}
for some $c_{\ell,\alpha} \in \mathbb{C}$.
\end{lema}

\begin{proof}
The proof is by induction on $d$. The claim for $d=0$ is obvious. Let us assume that for any polynomial  $p(\mathbf x)$ of degree at most $d$ and any positive integer $L$ the function $p(\mathbf x)g_L(\mathbf x)$ can be written in the form~\eqref{eg:polynomial_form}. We will prove the claim for any polynomial $q(\mathbf x)$ of degree $d+1$ and any $L \in \mathbb{N}$. By linearity, it is enough to prove the claim for $q(\mathbf{x})=x_jp(\mathbf{x})$  with $1 \leq j \leq N$ and $p \in \mathbb{P}_d$. Since $g_{L+1} \in C^1(\mathbb{R}^N)$ is radial and $p \in C^1(\mathbb{R}^N)$, by~\eqref{eq:Leibniz} we have
\begin{equation}\label{eq:Leibniz_application}
\begin{split}
T_{j}(pg_{L+1})(\mathbf{x})&=p(\mathbf{x})T_{j}g_{L+1}(\mathbf{x})+g_{L+1}(\mathbf{x})T_{j}p(\mathbf{x})\\&=2(L+1)x_jp(\mathbf{x})g_{L}(\mathbf{x})+g_{L+1}(\mathbf{x})T_{j}p(\mathbf{x}).
\end{split}
\end{equation}
By Proposition~\ref{propo:polynomial_degree} we have $T_{j}p \in \mathbb{P}_{d-1}$, so, by the induction hypothesis,
\begin{align*}
g_{L+1}(\mathbf{x})T_{j}p(\mathbf{x})&=\sum_{\ell=0}^{d-1}\sum_{\|\alpha\| \leq \ell}c_{\ell,\alpha}T^{\alpha}(g_{L+1+\ell})(\mathbf{x})=\sum_{\ell=0}^{d+1}\sum_{\|\alpha\| \leq \ell}c_{\ell,\alpha}'T^{\alpha}(g_{L+\ell})(\mathbf{x}).
\end{align*}
Therefore, by~\eqref{eq:Leibniz_application}, it is enough to check that $T_{j}(pg_{L+1})$ can be written in the form~\eqref{eg:polynomial_form}.  Since $p \in \mathbb{P}_d$, by the induction hypothesis
\begin{align*}
p(\mathbf{x})g_{L+1}(\mathbf{x})=\sum_{\ell=0}^{d}\sum_{\|\alpha\| \leq \ell}d_{\ell,\alpha}T^{\alpha}(g_{L+1+\ell})(\mathbf{x}),
\end{align*}
therefore
\begin{align*}
T_j(pg_{L+1})(\mathbf{x})&=\sum_{\ell=0}^{d}\sum_{\|\alpha\| \leq \ell}d_{\ell,\alpha}T_j \circ T^{\alpha}(g_{L+1+\ell})(\mathbf{x})\\&=\sum_{\ell=0}^{d}\sum_{\|\alpha\| \leq \ell}d_{\ell,\alpha}T^{\alpha+e_j}(g_{L+1+\ell})(\mathbf{x})=\sum_{\ell=0}^{d+1}\sum_{\|\alpha\| \leq \ell}d_{\ell,\alpha}'T^{\alpha}(g_{L+\ell})(\mathbf{x}).
\end{align*}
\end{proof}

\begin{lema}\label{lem:approx}
The set $\bigcup_{d \in \mathbb{N}_0}\bigcup_{p \in \mathbb{P}_d}\{p(\cdot)g_1(\cdot)\}$ is dense in $L^2(B(0,1),dw)$.
\end{lema}

\begin{proof}
Take any $f \in L^2(B(0,1),dw)$ and fix $\varepsilon>0$. There is $\delta >0$ such that
\begin{align*}
\|f-f\chi_{B(0,1-\delta)}\|_{L^2(dw)}<\varepsilon.
\end{align*}
Let $\Phi \in C^{\infty}_{c}(B(0,1))$ be a radial function such that $\int \Phi(\mathbf{x})\,dw(\mathbf{x})=1$. Let us denote $\Phi_t(\mathbf{x})=t^{-\mathbf{N}}\Phi(t^{-1}\mathbf{x})$. Since $\Phi_t$ is an approximate of the identity on $L^2(dw)$, there is $t>0$ such that
\begin{align*}
\|\Phi_t*(f\chi_{B(0,1-\delta)})-f\chi_{B(0,1-\delta)}\|_{L^2(dw)}<\varepsilon
\end{align*}
and $\supp \Phi_t*(f \chi_{B(0,1-\delta)}) \subseteq B(0,1-\delta/2)$. Since $f\chi_{B(0,1-\delta)} \in L^2(dw)$, the function $\Phi_t*(f\chi_{B(0,1-\delta)})$ is continuous. Moreover, there is $\eta>0$ such that $g_1(\mathbf{x})>\eta$ for  $\mathbf{x} \in B(0,1-\delta/2)$. This implies that $(\Phi_t*(f \chi_{B(0,1-\delta)}))g_1^{-1}$ is a continuous function on $B(0,1-\delta/2)$, which extends to a continuous $h$ function on $B(0,1)$ (by putting  $0$ on $B(0,1) \setminus B(0,1-\delta/2)$). Therefore, by the Stone--Weierstrass theorem,  there is a polynomial $p$ such that
\begin{equation}\label{eq:Phi_approx}
\|h-p\|_{L^{\infty}(B(0,1))}<\varepsilon.
\end{equation}
Finally, by~\eqref{eq:Phi_approx},
\begin{align*}
\|\Phi_t*(f\chi_{B(0,1-\delta)})-pg_1\|_{L^2(dw)}& = \|hg_1-pg_1\|_{L^2(dw)} \\
&\leq w(B(0,1))^{1\slash 2} \|hg_1-pg_1\|_{L^\infty(B(0,1))}\\
&\leq w(B(0,1))^{1\slash 2}\|g_1\|_{L^{\infty}(B(0,1))}\|h-p\|_{L^\infty(B(0,1))}\\
&\leq w(B(0,1))^{1\slash 2}\varepsilon.
\end{align*}
\end{proof}

\begin{proof}[Proof of Theorem \ref{teo:support}]
It suffices to consider $r=1$. Let $f \in L^2(B(0,1),dw)$. Fix $\varepsilon>0$. By Lemma~\ref{lem:approx} there is a polynomial $p$ such that
\begin{align*}
\|f-pg_1\|_{L^2(dw)}<\varepsilon.
\end{align*}
Since the Dunkl translation is bounded on $L^2(dw)$ and its norm is $1$, we have
\begin{equation}\label{eq:epsilon}
\|\tau_{\mathbf{x}}(f-(pg_1))\|_{L^2(dw)}=\|\tau_{\mathbf{x}}f-\tau_{\mathbf{x}}(pg_1)\|_{L^2(dw)}<\varepsilon.
\end{equation}
By Lemma~\ref{lem:representation} and the fact that the Dunkl translations commute with the Dunkl operators, the function $\tau_{\mathbf{x}}(pg_1)$ can be written in the form
\begin{align*}
\tau_{\mathbf{x}}(pg_{1})(-\mathbf{y})&=\tau_{\mathbf{x}}\left(\sum_{\ell=0}^{d}\sum_{\|\alpha\| \leq \ell }c_{\ell,\alpha}T^{\alpha}(g_{1+\ell})(-\mathbf{y})\right)=\sum_{\ell=0}^{d}\sum_{\|\alpha\| \leq \ell }c_{\ell,\alpha}T^{\alpha}\tau_{\mathbf{x}}(g_{1+\ell})(-\mathbf{y}),
\end{align*}
where $d$ is the degree of $p$. Since the functions $g_{1+\ell}$ are radial and supported by $B(0,1)$, by~\eqref{eq:support_radial} we have that $\supp \tau_{\mathbf x} g_{\ell+1}(-\, \cdot )  \subseteq \mathcal{O}(B(\mathbf{x},1))$. This implies that $\supp ( T ^{\alpha} \tau_{\mathbf x} g_{\ell+1})(-\, \cdot) \subseteq \mathcal{O}(B(\mathbf{x},1))$ and, finally,
\begin{equation}\label{eq:support}
\supp \tau_{\mathbf x}(pg_1)(-\, \cdot)  \subseteq \mathcal{O}(B(\mathbf{x},1)).
\end{equation}
Since $\varepsilon>0$ is taken arbitrarily,~\eqref{eq:epsilon} and~\eqref{eq:support} imply the claim.
\end{proof}

\section{Consequences of Theorem~\ref{teo:support}}

\begin{coro}
Suppose that $\mathbf{x},\mathbf{y} \in \mathbb{R}^N$ satisfy $\|\mathbf{x}\|+\|\mathbf{y}\| < 1$. If $f$ is a continuous compactly supported function such that $f(\mathbf{z})=0$ for all $\mathbf{z} \in B(0,1)$, then $\tau_{\mathbf{x}}f(\mathbf{y})=0$.
\end{coro}

\begin{proof} The corollary follows from~\cite[Theorem 5.1]{AAS-Colloq}. We present here an alternative proof.

Take $\varepsilon>0$ such that $\|\mathbf{x}\|+\|\mathbf{y}\|+\varepsilon<1$.  Let $g\in L^2(dw)$, $\text{\rm supp}\, g\subseteq B(0, \| \mathbf y\|+\varepsilon)$. We have
\begin{equation}\label{eq:translation_duality}
\int \tau_{\mathbf{x}}f(\mathbf{z}) g(\mathbf{z})\,dw(\mathbf{z})=\int f(\mathbf{z})\tau_{-\mathbf{x}} g(\mathbf{z})\,dw(\mathbf{z}).
\end{equation}
 By Theorem~\ref{teo:support}, 
\begin{align*}
\supp \tau_{-\mathbf{x}} g \subseteq \mathcal{O}(B(\mathbf{x},\|\mathbf{y}\|+\varepsilon))\subseteq B(0,\|\mathbf{x}\|+\|\mathbf{y}\|+\varepsilon) \subseteq B(0,1).
\end{align*}
By our assumption $f(\mathbf{z})=0$ for all $\mathbf{z} \in B(0,1)\supseteq \supp \tau_{-\mathbf{x}}g$, so the second integral in~\eqref{eq:translation_duality} is zero. Thus $\tau_{\mathbf x}f \equiv 0$ on $B(0,\| y\|+\varepsilon)$. In particular, $\tau_{\mathbf{x}}f(\mathbf{y})=0$.
\end{proof}
\begin{lema}\label{lem:L2_better}
There is a constant $C>0$ such that for any $r>0$, $\mathbf{x} \in \mathbb{R}^N$, and any radial function  $\phi\in C_c (B(0,r))$ we have
\begin{align*}
\|\tau_{\mathbf{x}}\phi(-\, \cdot)\|_{L^2(dw)} \leq C \frac{r^{\mathbf N} \| \phi\|_{L^\infty}}{w(B(\mathbf{x},r))^{1/2}}.
\end{align*}
\end{lema}

\begin{proof}
By~\eqref{eq:support_radial}  $\supp \tau_{\mathbf x}\phi(-\, \cdot) \subseteq \mathcal{O}(B(\mathbf{x},r))$, so
\begin{align*}
\|\tau_\mathbf{x}\phi\|_{L^2(dw)} \leq |G|^{1/2}w(B(\mathbf{x},r))^{1/2}\|\tau_\mathbf{x}\phi\|_{L^{\infty}}.
\end{align*}
Furthermore, by~\cite[Corollary 3.10]{conjugate}, there is a constant $C>0$ such that
\begin{align*}
 \|\tau_\mathbf{x}\phi\|_{L^{\infty}} \leq C\frac{r^{\mathbf N}\| \phi\|_{L^\infty} }{w(B(\mathbf{x},r))},
\end{align*}
so the lemma follows.
\end{proof}

\begin{propo}\label{propo:compact_supports}
There is a constant $C>0$ such that for any  $r_1,r_2>0$, any $f\in L^1(dw)$ such that $\supp f \subseteq B(0,r_2)$, any radial function $\phi\in C_c(B(0,r_1))$, and for all $\mathbf{y} \in \mathbb{R}^N$ we have
\begin{align*}
\|\tau_{\mathbf{y}}(f * \phi)\|_{L^1(dw)} \leq C (r_1(r_1+r_2))^{\frac{\mathbf{N}}{2}}
\| \phi\|_{L^{\infty}}\|f\|_{L^1(dw)}.
\end{align*}
\end{propo}
\begin{proof}
By Theorem~\ref{teo:support}, $\supp \tau_{\mathbf{y}}(f * \phi)(-\, \cdot) \subseteq \mathcal{O}(B(\mathbf{y},r_1+r_2))$. Therefore
\begin{equation}\label{eq:L1_by_L2}
\|\tau_{\mathbf{y}}(f * \phi)\|_{L^1(dw)} \leq |G|^{1/2}w(B(\mathbf{y},r_1+r_2))^{1/2} \|\tau_{\mathbf{y}}(f * \phi)\|_{L^2(dw)}.
\end{equation}
Since $\tau_{\mathbf{y}}(f * \phi)=f * (\tau_{\mathbf{y}}\phi)$, we have
\begin{align*}
\|\tau_{\mathbf{y}}(f * \phi)\|_{L^2(dw)}&=\left(\int|(f*\tau_{\mathbf{y}}\phi)(\mathbf{x})|^2\,dw(\mathbf{x})\right)^{1/2}\\
&=\left(\int \left|\int \tau_{\mathbf{x}}(\tau_{\mathbf{y}}\phi)
(-\mathbf{z})f(\mathbf{z})\,dw(\mathbf{z})\right|^2\,dw(\mathbf{x})\right)^{1/2}\\
&=\left(\int \left|\int \tau_{-\mathbf{z}}(\tau_{\mathbf{y}}\phi)(\mathbf{x})f(\mathbf{z})\,dw(\mathbf{z})\right|^2\,dw(\mathbf{x})\right)^{1/2}.
\end{align*}
By Minkowski integral inequality
\begin{equation}\label{eq:Minkowski}
\left(\int \left|\int \tau_{-\mathbf{z}}(\tau_{\mathbf{y}}\phi)(\mathbf{x})f(\mathbf{z})\,dw(\mathbf{z})\right|^2\,dw(\mathbf{x})\right)^{1/2}
\leq \int |f(\mathbf{z})| \|\tau_{-\mathbf{z}}\tau_{\mathbf{y}}\phi\|_{L^2(dw)} \,dw(\mathbf{z}).
\end{equation}
Since $g \mapsto \tau_{-\mathbf{z}}g$ is a contraction on $L^2(dw)$ for all $\mathbf{z} \in \mathbb{R}^N$, by Lemma~\ref{lem:L2_better} we have
\begin{equation}\label{eq:double_translate}
\|\tau_{-\mathbf{z}}\tau_{\mathbf{y}}\phi\|_{L^2(dw)} \leq \|\tau_{\mathbf{y}}\phi\|_{L^2(dw)} \leq C\frac{r_1^{\mathbf N} \| \phi\|_{L^\infty}}{w(B(\mathbf{y},r_1))^{1/2}}.
\end{equation}
Therefore, by~\eqref{eq:Minkowski} and~\eqref{eq:double_translate},
\begin{align*}
\|\tau_{\mathbf{y}}(f * \phi)\|_{L^2(dw)}
\leq C\frac{r_1^{\mathbf N} \| \phi\|_{L^\infty} }{w(B(\mathbf{y},r_1))^{1/2}}
\|f\|_{L^1(dw)}.
\end{align*}
Finally, by~\eqref{eq:L1_by_L2},
\begin{align*}
\|\tau_{\mathbf{y}}(f * \phi)\|_{L^1(dw)} & \leq C |G|^{1/2} w(B(\mathbf{y},r_1+r_2))^{1/2}
\frac{r_1^{\mathbf N} \|\phi\|_{L^\infty}}{w(B(\mathbf{y},r_1))^{1\slash 2}} \|f\|_{L^1(dw)} \\
&\leq C' \left(r_1(r_1+r_2)\right)^{\mathbf{N}/2}\| \phi\|_{L^\infty} \|f\|_{L^1(dw)}.
\end{align*}
\end{proof}

Let $\Psi_0 \in C^{\infty}((-1,1))$ and $\Psi \in C^{\infty}(\frac{1}{4},4)$ be such that 
\begin{align}\label{eq:reproduce}
1=\Psi_{0}(\|\mathbf{x}\|)+\sum_{n=1}^{\infty}\Psi(2^{-n}\|\mathbf{x}\|)=\sum_{n=0}^{\infty}\Psi_{n}(\|\mathbf{x}\|) \text{ for all }\mathbf{x}\in \mathbb{R}^N.
\end{align}

\begin{propo}\label{propo:weights} Fix  $\delta\geq 0$.  Assume that $\phi$ is a continuous radial function such that
\begin{equation}\label{eq:A1}
\sum_{n=0}^\infty 2^{n(\mathbf N+\delta)} \| \phi(\cdot)\Psi_n(\|\cdot\|)\|_{L^\infty}=A<\infty
\end{equation}
and $f$ is a measurable function on $\mathbb R^N$ such that
\begin{equation}\label{eq:B1}
\sum_{j=0}^\infty 2^{j(\mathbf N\slash 2+\delta)}\| f(\cdot)\Psi_n(\|\cdot\|)\|_{L^1(dw)} =B<\infty.
\end{equation}
Then $f*\phi\in L^2(dw)\cap L^1(dw)$ and there is a constant $C>0$ such that for every $\mathbf y\in\mathbb R^N$ we have
\begin{equation}
\int |\tau_{\mathbf y}(f*\phi)(-\mathbf x)|(1+d(\mathbf x,\mathbf y))^\delta \, dw(\mathbf x)
\leq C A B.
\end{equation}
\end{propo}

\begin{proof} In the proof we will use the formula
$$\tau_{\mathbf x}(f*\phi)(-\mathbf y)=\tau_{-\mathbf y}(f*\phi) (\mathbf x)=(f*\phi)(\mathbf x,\mathbf y).$$
Let $f_j(\mathbf{x})=f(\mathbf{x})\Psi_j(\|\mathbf{x}\|) $, $\phi_n(\mathbf{x})=\phi(\mathbf{x})\Psi_n(\|\mathbf{x}\|)$. Observe that $f=\sum_{j=0}^\infty f_j$ and the series converges in $L^1(dw)$. Moreover, $\| \phi_n \|_{L^2(dw)} \leq C2^{n\mathbf N/2}\| \phi_n\|_{L^\infty}$, hence $\phi=\sum_{n=0}^\infty \phi_n$  and the  convergence is in $L^2(dw)$. So, by \eqref{eq:t2},  $f*\phi\in L^2(dw)$. Further, the double series
   $$ f*\phi =\sum_{j,n \in \mathbb{N}_0} f_j*\phi_n$$
  is absolutely convergent in $L^2(dw)$, because
   $$ \sum_{j,n \in \mathbb{N}_0} \| f_j*\phi_n\|_{L^2(dw)} \leq \sum_{j,n \in \mathbb{N}_0} \| f_j\|_{L^1(dw)}\| \phi_n\|_{L^2(dw)}\leq C\sum_{j,n \in \mathbb{N}_0} \| f_j\|_{L^1(dw)} 2^{n\mathbf N/2}\| \phi_n\|_{L^\infty}\leq CAB.$$
 Using \eqref{eq:t1} and \eqref{eq:t2} we have
\begin{equation}\label{eq:conv1}
\tau_{-\mathbf y}(f*\phi)=\sum_{j,n \in \mathbb{N}_0} \tau_{-\mathbf y} (f_j*\phi_n)=\sum_{j,n \in \mathbb{N}_0} f_j*(\tau_{-\mathbf y}\phi_n)
\end{equation}and the convergence (absolute) is in $L^2(dw)$. Clearly, $\supp\, f_j*\phi_n \subseteq B(0,2^{j}+2^{n}) $. Theorem~\ref{teo:support} implies $\tau_{-\mathbf y} (f_j*\phi_n)(\mathbf x)=0 $ for $d(\mathbf x,\mathbf y)>2^{j}+2^{n}$. Using  Proposition~\ref{propo:compact_supports}  we obtain
\begin{equation*}\begin{split}
\sum_{j,n \in \mathbb{N}_0} \int |(f_j * \phi_n)(\mathbf x,\mathbf y) & |(1+d(\mathbf x,\mathbf y))^\delta\, dw(\mathbf x) \\
&\leq C \sum_{j,n \in \mathbb{N}_0} 2^{\mathbf{N}n/2}(2^j+2^n)^{\delta+\mathbf N\slash 2} \|f_j\|_{L^1(dw)}\|\phi_n\|_{L^\infty} \leq CAB.
\end{split}\end{equation*}
Thus, the double series~\eqref{eq:conv1} converges in the $L^1((1+d(\mathbf x,\mathbf y))^{\delta}\, dw(\mathbf x))$-norm as well. The proof of the proposition is complete.
\end{proof} 

\begin{coro}\label{coro:weight}
Assume that there is $\delta>0$ such that  $f(\mathbf x)(1+\| \mathbf x\|)^{\mathbf N\slash 2+\delta} \in L^1(dw)$ and $ \phi(\mathbf x)(1+\| \mathbf x\|)^{\mathbf N+\delta}\in L^\infty(dw)$. Then for every $0<\delta'<\delta$ there is a constant $C=C_{\delta,\delta'}$ such that
\begin{equation}\begin{split} \| \tau_{-\mathbf  y}(f*\phi)(\mathbf x)
&(1+ d(\mathbf x,\mathbf y))^{\delta'}\|_{L^1 dw(\mathbf x)}\\
&\leq C \|f(\cdot)(1+\|\cdot\|)^{\mathbf{N}/2+\delta}\|_{L^1(dw)}\|\phi(\cdot)(1+\|\cdot\|)^{\mathbf{N}+\delta}\|_{L^{\infty}}.
\end{split}\end{equation}
In particular,
\begin{equation}\label{eq:weights_app}
\begin{split}
\int_{\mathcal{O}(B(\mathbf{y},r))^{c}} & |\tau_{-\mathbf y}(f*\phi)(\mathbf x)| \, dw(\mathbf x)
 \\&\leq Cr^{-\delta'}\|f(\cdot)(1+\|\cdot\|)^{\mathbf{N}/2+\delta}\|_{L^1(dw)}\|\phi(\cdot)(1+\|\cdot\|)^{\mathbf{N}+\delta}\|_{L^{\infty}}.
\end{split}
\end{equation}
\end{coro}

 \begin{coro}\label{coro:weight2}
Let  $s>\mathbf N$ be a positive integer and  $\varepsilon >0$.  Assume that a  function $F\in C^{2s}(\mathbb R^N)\cap L^1(dw)$ satisfies: 
$$ B_1=\| (I-\Delta)^s F(\mathbf x)(1+\| x\|)^{\mathbf N\slash 2+\varepsilon}\|_{L^1(dw(\mathbf x))}<\infty.$$
Then there is a constant $C_{s,\varepsilon}>0$ such that
$$ \sup_{\mathbf y \in \mathbb R^N}\| \tau_{\mathbf y}F\|_{L^1(dw)}\leq C_{s,\varepsilon} B_1. $$
In particular, for every $1\leq p< \infty$ we have  $\| g*F\|_{L^p(dw)}\leq C_{s,\varepsilon} B_1\| g\|_{L^p(dw)}$.
\end{coro}
\begin{proof}
Set $f=(I-\Delta)^{s}F$, $g(x)=c_k^{-1}\mathcal F^{-1}\{ (1+\| \xi\|^2)^{-s}\}(x)$. Then $g$ is a radial function satisfying $|g(\mathbf x)|\leq C_L (1+\| \mathbf x\|)^{-L}$  for every $L>0$. Clearly, $F=f*g$. Thus the corollary is a direct consequence of Proposition \ref{propo:weights}.
\end{proof}
\begin{coro}\label{coro:Schart_transl}
For every $F \in \mathcal S(\mathbb R^N)$ there is a constant $C>0$ such that
$$\sup_{\mathbf y\in \mathbb R^N} \| \tau_{\mathbf y}F\|_{L^1(dw)}\leq C.$$
\end{coro}

\begin{teo}
Let $\Phi \in \mathcal{S}(\mathbb{R}^N)$. Then the maximal function 
\begin{align*}
\mathcal{M}_{\Phi}f(\mathbf{x})=\sup_{t>0}|\Phi_t * f(\mathbf{x})|=\sup_{t>0}\left|\int \Phi_t(\mathbf{x},\mathbf{y})f(\mathbf{y})\,dw(\mathbf{y})\right|,
\end{align*}
where $\Phi_t(\mathbf{x})=t^{-\mathbf{N}}\Phi(t^{-1}\mathbf{x})$, is of weak type $(1,1)$ and bounded on $L^p(dw)$ for $1<p\leq\infty$.
\end{teo}

\begin{proof}
It is enough to prove that there is a constant $C=C_{\Phi}>0$ such that $$\mathcal{M}_{\Phi}f(\mathbf{x}) \leq C\sum_{\sigma\in G}\mathcal{M}_{\rm HL}f(\sigma(\mathbf{x})),$$ where $\mathcal{M}_{\rm HL}$ is Hardy--Littlewood maximal function on the space of homogeneous type $(\mathbb R^N, \|\mathbf x-\mathbf y\|, dw)$. To this end it suffices to prove that there are constants $C,\delta>0$ such that for all $\mathbf{x},\mathbf{y} \in \mathbb{R}^N$ we have
\begin{equation}\label{eq:convolution_goal}
|\Phi_t(\mathbf{x},\mathbf{y})| \leq Cw(B(\mathbf{x},t))^{-1}\Big(1+\frac{d(\mathbf{x},\mathbf{y})}{t}\Big)^{-\mathbf{N}-\delta}.
\end{equation}
Let $
g(\mathbf{x})=c_k\mathcal{F}^{-1}((1+\|\cdot\|^2)^{-\mathbf{N}}) 
$. 
The function $g$ is a radial and  satisfies  $|g(\mathbf{x})| \leq C_{1}(1+\|\mathbf{x}\|)^{-2\mathbf{N}-\delta}$, so by~\cite[Corollary 3.10]{conjugate}, we have
\begin{equation}\label{eq:Corollary310_app}
|g_t(\mathbf{x},\mathbf{z})| \leq C_2w(B(\mathbf{x},t))^{-1}\Big(1+\frac{d(\mathbf{x},\mathbf{z})}{t}\Big)^{-s} \text{ for }s\in\{0,\mathbf{N}+\delta\}.
\end{equation}
Set $\Phi^{\{1\}}=\mathcal{F}^{-1}((\mathcal{F}\Phi)(1+\|\cdot\|^2)^{2\mathbf{N}})$.
Then $\Phi^{\{1\}}\in\mathcal S(\mathbb R^N)$, $\Phi_t=\Phi^{\{1\}}_t*g_t*g_t$, and
\begin{align*}
|\Phi_{t}(\mathbf{x},\mathbf{y})| \leq \int \Big|g_t(\mathbf{x},\mathbf{z})(\Phi^{\{1\}}_t*g_t)(\mathbf{z},\mathbf{y})\Big|\,dw(\mathbf{z}) \leq \int_{d(\mathbf{x},\mathbf{y}) \leq 2d(\mathbf{x},\mathbf{z})}+\int_{d(\mathbf{x},\mathbf{y}) \leq 2d(\mathbf{y},\mathbf{z})}=I_1+I_2.
\end{align*}
Now,~\eqref{eq:Corollary310_app} with $s=\mathbf{N}+\delta$ and Corollary~\ref{coro:weight} with $\delta'=0$ lead  to
\begin{align*}
I_1 &\leq C_2w(B(\mathbf{x},t))^{-1}\int_{d(\mathbf{x},\mathbf{y}) \leq 2d(\mathbf{x},\mathbf{z})} \Big(1+\frac{d(\mathbf{x},\mathbf{z})}{t}\Big)^{-\mathbf{N}-\delta}|(\Phi^{\{1\}}_t*g_t)(\mathbf{z},\mathbf{y})|\,dw(\mathbf{z})
\\ &\leq C_3w(B(\mathbf{x},t))^{-1}\Big(1+\frac{d(\mathbf{x},\mathbf{y})}{t}\Big)^{-\mathbf{N}-\delta}\int |(\Phi^{\{1\}}_t*g_t)(\mathbf{z},\mathbf{y})|\,dw(\mathbf{z}) \\& \leq C_4 w(B(\mathbf{x},t))^{-1}\Big(1+\frac{d(\mathbf{x},\mathbf{y})}{t}\Big)^{-\mathbf{N}-\delta}.
\end{align*}
Further,
\begin{align*}
I_2 \leq  C_5\Big(1+\frac{d(\mathbf{x},\mathbf{y})}{t}\Big)^{-\mathbf{N}-\delta}\int|g_t(\mathbf{x},\mathbf{z})(\Phi^{\{1\}}_t*g_t)(\mathbf{z},\mathbf{y})|\Big(1+\frac{d(\mathbf{y},\mathbf{z})}{t}\Big)^{\mathbf{N}+\delta}\,dw(\mathbf{z}),
\end{align*}
so by~\eqref{eq:Corollary310_app} with $s=0$ and Corollary~\ref{coro:weight} with $\delta'=\mathbf{N}+\delta$, we obtain~\eqref{eq:convolution_goal}.
\end{proof}

\section{Multipliers}\label{sec:multipliers}
Let $m$ be a function defined on $\mathbb R^N$. In this section we
assume that there exists $s>\mathbf N$ such that the multiplier $m$
satisfies~\eqref{eq:assumption}. 

Fix  $\phi$ a radial $C^\infty$ function on $\mathbb R^N$ supported in the annulus $\{\xi\in\mathbb R^N: 1\slash 2\leq \| \xi\|\leq 2  \}$ such that 
$1=\sum_{\ell \in \mathbb{Z}}\phi(2^{-\ell}\xi)$. We define $m_\ell(\xi)$, $m_{\ell,1}(\xi)$, and $m_{\ell, 2}(\xi)$ as follows: 
\begin{equation}\label{eq:function_m_l} m_{\ell}(\xi)=m(2^{\ell}\xi)\phi(\xi)=m_{\ell, 1}(\xi)e^{-\|\xi\|^2}=m_{\ell,2}(\xi)e^{-\|\xi\|^2}e^{-\|\xi\|^2}.
\end{equation}
{
By assumption~\eqref{eq:assumption} there is $C>0$ such that 
$$\sup_{\ell\in \mathbb Z}  \| m_\ell\|_{W^s_2}\leq CM. $$
Proposition 5.3 of \cite{ABDH} asserts that for any real numbers $\alpha >\beta >0$ there is a constant $C=C_{\alpha,\beta}$ such that 
\begin{equation}\label{eq:insertion}
\| \mathcal Fm_{\ell} (\mathbf x)(1+\| \mathbf x\|)^\beta \|_{L^2(dw(\mathbf x))} \leq C \| m_{\ell} \|_{W^\alpha_2} = C \| \widehat m_{\ell}(\mathbf x) (1+\| \mathbf x\|)^{\alpha} \|_{L^2(d\mathbf x)},\end{equation}
where $\widehat m_{\ell}$ denotes the classical Fourier transform of $m_{\ell}$.

Set
\begin{equation}\label{eq:tilde_K}
\tilde K_\ell (\mathbf x,\mathbf y)=\tau_{-\mathbf y} (\mathcal F^{-1} m_\ell )(\mathbf x),
\end{equation}
\begin{equation}\label{eq:just_K}
K_\ell (\mathbf{x},\mathbf{y})=\tau_{-\mathbf{y}}\mathcal{F}^{-1}(m(\cdot)\phi(2^{-\ell} \cdot))(\mathbf{x}).
\end{equation}
By homogeneity, 
\begin{equation}\label{eq:KK} K_\ell (\mathbf x,\mathbf  y)=2^{\mathbf N\ell}\tilde K_\ell (2^\ell \mathbf x, 2^{\ell} \mathbf y).
\end{equation}
Obviously, $\tilde K_\ell (\mathbf x,\mathbf y) $ and $K_\ell(\mathbf x,\mathbf y)$ are $C^\infty$ functions of $\mathbf x$, $\mathbf y$, since $m_\ell$ is, by assumption~\eqref{eq:assumption}, a bounded compactly supported function.  

Let $\mathcal{T}_m$ and $\mathcal{T}_{\ell}$ denote the Dunkl multiplier operators associated with $m$ and $m_\ell(2^{-\ell}\cdot)=\phi(2^{-\ell} \cdot)m(\cdot )$ respectively. 
Obviously, for $f\in L^2(dw)$ one has
\begin{equation}\label{eq:kernel_for_cutoff}
\mathcal{T}_\ell f(\mathbf x)=\int K_\ell (\mathbf x,\mathbf y) f(\mathbf y)\, dw(\mathbf y).
\end{equation}
Clearly, $\| m_{\ell, 1}\|_{W^s_2}\leq CM$. Using \eqref{eq:insertion} and the Cauchy--Schwarz inequality  together with \eqref{eq:homo},  we deduce that for every  $\delta'\geq 0$ such that  $\mathbf N+\delta'< s $ we have 
$$ \| \mathcal F^{-1}m_{\ell, 1}(\mathbf x)(1+\| \mathbf x\|)^{\mathbf N\slash 2+\delta'} \|_{L^1(dw(\mathbf x))}\leq C_{\delta', \alpha}\| m_{\ell,1} \|_{W^s_2}\leq CM.$$
Recall that
\begin{equation}\label{eq:cnv} \mathcal F^{-1} m_\ell (\mathbf x)= (\mathcal F^{-1} m_{\ell, 1})* h_1(\mathbf x),
\end{equation}
where $h_1(\mathbf x)=2^{-\mathbf{N}/2}c_k^{-1}\exp (-\| \mathbf x\|^2\slash 4)$ (see~\eqref{eq:heat_semigroup}).
Therefore, by Corollary~\ref{coro:weight} together with~\eqref{eq:insertion} we have
\begin{equation}\label{eq:m_ell} \begin{split}\int |\tilde K_\ell (\mathbf x,\mathbf y)|(1+d(\mathbf x,\mathbf y))^{\delta'}\,dw(\mathbf x) & =
\int |\tau_{-\mathbf y} (\mathcal{F}^{-1}m_\ell )(\mathbf x)|(1+ d(\mathbf x,\mathbf y))^{\delta'}\, dw(\mathbf x)\\& \leq C_{\delta', s} \| m_{\ell, 1}\|_{W^s_2}\leq CM.
\end{split}\end{equation}
By the same arguments with $\delta'=0$ we obtain 
\begin{equation}\label{eq:m_ell_1}
\int |\tau_{-\mathbf y}(\mathcal F^{-1}m_{\ell,1}) (\mathbf x)|\, dw(\mathbf x)\\
\leq C_{s,\delta'} \| m_{\ell, 2}\|_{W^s_2}\leq CM.
\end{equation}
From~\eqref{eq:KK} and~\eqref{eq:m_ell} we conclude
\begin{equation}\label{eq:K_ell_delta}
\int |K_\ell (\mathbf x,\mathbf y)|(1+d(\mathbf x,\mathbf y))^{\delta'} \, dw(\mathbf x)\leq CM 2^{-\delta'\ell}.
\end{equation}
On the other hand, using \eqref{eq:cnv}  together with \eqref{eq:m_ell_1} we get 
\begin{equation}\begin{split}\label{eq:Holder_tilde}
\int |\tilde K_\ell(\mathbf x,\mathbf y)& -\tilde K_{\ell}(\mathbf x,\mathbf y')|\, dw(\mathbf x)\\
&
= \int \Big|\int \tau_{-\mathbf z}(\mathcal F^{-1} m_{\ell ,1})(\mathbf x)\Big(h_1(\mathbf z,\mathbf y)-
h_1(\mathbf z,\mathbf y')\Big)\, dw(\mathbf z)\Big| \, dw(\mathbf x)\\
& \leq C\| m_{\ell , 2}\|_{W^s_2} \int \Big| h_1(\mathbf x,\mathbf y)-h_1(\mathbf z,\mathbf y')\Big|\, dw(\mathbf z)\\
&\leq C' M \| \mathbf y-\mathbf y'\|,
\end{split}
\end{equation}
where in the last inequality we have used~\eqref{eq:Holder}.
From \eqref{eq:KK} and \eqref{eq:Holder_tilde} we easily deduce
\begin{equation}\begin{split}\label{eq:Holder_int}
\int |K_\ell(\mathbf x,\mathbf y)& - K_{\ell}(\mathbf x,\mathbf y')|\, dw(\mathbf x)
\leq C M 2^\ell \| \mathbf y-\mathbf y'\|.
\end{split}
\end{equation}
For a cube $Q\subset \mathbb R^N$ let $c_Q$ be its center and  $\diam (Q)$ be the length of its diameter.  Let $Q^{*}$ denote the cube with center $c_Q$ such that $\diam(Q^{*})=2\diam(Q)$.  The following proposition is a direct consequence of   \eqref{eq:K_ell_delta} and \eqref{eq:Holder_int}.
\begin{propo}\label{propo:integral} There are constants $C, \delta'>0$ such that for every cube $Q\subset\mathbb R^N$ and $\mathbf y,\mathbf y'\in Q$ we have
\begin{equation*}
\int_{\mathbb R^N\setminus \mathcal O (Q^{*})} |K_\ell(\mathbf x,\mathbf y) - K_{\ell}(\mathbf x,\mathbf y')|\, dw(\mathbf x)
\leq C M \min \Big((2^{\ell} \text{\rm diam}(Q))^{-\delta'} , 2^\ell \text{\rm diam}(Q)\Big).
\end{equation*}
\end{propo}
}

\section{Proof of Theorem\texorpdfstring{ ~\ref{teo:main_teo}~\eqref{numitem:weak_type} and~\eqref{numitem:strong_type}}{1.1 (A)and (B)}}\label{sec:section_proof}
\begin{proof} Having Proposition \ref{propo:integral} already established, the proof of weak type $(1,1)$ of the multiplier operator $\mathcal T_m$ follows the standard pattern. 
Clearly, there is a constant $C_1>1$, which depends on the doubling constant and $N$, such that
$w(Q)\leq C_1 w(Q')$, where $Q'$ is any sub-cube of $Q$, $\ell (Q')=\ell (Q)\slash 2$, where $\ell (Q)$ denote the side length of $Q$.
\\
Let $f\in L^1(dw)\cap L^2(dw)$. Fix $\lambda>0$. Denote by $\mathcal Q_\lambda$ the collection of all maximal (disjoint)  dyadic cubes $Q_j$ in $\mathbb R^N$ satisfying
 $$ \lambda< \frac{1}{w(Q_j)}\int_{Q_j}|f(\mathbf{x})|\,dw(\mathbf{x}).$$
 Then 
$$ \frac{1}{w(Q_j)} \int_{Q_j}|f(\mathbf{x})|\,dw(\mathbf{x})\leq C_1\lambda.$$
   Set $\Omega=\bigcup_{Q_j\in\mathcal Q_\lambda} Q_j$. Then $w(\Omega)\leq \lambda^{-1}\| f\|_{L^1(dw)}$.  Form the corresponding Calder\'on--Zygmund decomposition of $f$, namely, $f=g+b$, where

\begin{equation*}\label{eq:good}
g(\mathbf{x})=f\chi_{\Omega^c}(\mathbf{x})+\sum_{j}w(Q_j)^{-1}\Big(\int_{Q_j}f(\mathbf{y})\,dw(\mathbf{y})\Big)\chi_{Q_j}(\mathbf{x}),
\end{equation*}
\begin{equation*}\label{eq:bad}
b(\mathbf{x})=\sum_{j}b_j(\mathbf{x}), \text{ where }b_j(\mathbf{x})=\left(f(\mathbf{x})-w(Q_j)^{-1}\int_{Q_j}f(\mathbf{y})\,dw(\mathbf{y})\right)\chi_{Q_j}(\mathbf{x}).
\end{equation*}
Clearly, $g,b\in L^1(dw)\cap L^2(dw)$, $|g(\mathbf x)|\leq C_1\lambda$, $\| g\|_{L^2(dw)}^2 \leq C\lambda \| f\|_{L^1}$, $\sum_j \| b_j\|_{L^1(dw)}\leq C\| f\|_{L^1(dw)}$.  
Further,
\begin{align*}
w(\{\mathbf{x} \in \mathbb{R}^N\,:\,|\mathcal{T}_{m}f(\mathbf{x})|> \lambda) &\leq w(\{\mathbf{x} \in \mathbb{R}^N\,:\,|\mathcal{T}_{m}g(\mathbf{x})|>\lambda/2\})\\&+w(\{\mathbf{x} \in \mathbb{R}^N\,:\,|\mathcal{T}_{m}b(\mathbf{x})|>\lambda/2\}).
\end{align*}
Since $\mathcal{T}_m$ is bounded on $L^2(dw)$,  we obtain
\begin{align*}
w(\{\mathbf{x} \in \mathbb{R}^N\,:\,|\mathcal{T}_{m}g(\mathbf{x})|>\lambda/2\}) \leq \frac{4}{\lambda^{2}}\|m\|_{L^{\infty}}\|g\|_{L^2(dw)}^2\leq \frac{4C_1}{\lambda}\|m\|_{L^{\infty}}\|f\|_{L^1(dw)}.
\end{align*}
Let $Q_j^{*}$ be the cube with the same center $c_{Q_j}$ as $Q_j$ and two times larger side-length. Define $\Omega^{*}=\mathcal{O}\Big(\bigcup_{j}Q_j^{*}\Big)$. There is a constant $C_2>1$, which depends on the Weyl group, doubling constant, and $N$ such that
\begin{align*}
w(\Omega^{*}) \leq C_2w(\Omega) \leq C_2\lambda^{-1}\|f\|_{L^1(dw)}.
\end{align*}
Thus it suffices to estimate $\mathcal T_m b(\mathbf x)$ on $\mathbb R^N\setminus \Omega^*$. Since $\sum_{j}b_j$ converges to $b$ in $L^2(dw)$ and $\mathcal{T}_{m}b=\sum_{\ell \in \mathbb{Z}}\mathcal{T}_\ell b$ in the $L^2(dw)$-norm, we have
\begin{align*}
|\mathcal{T}_{m}b(\mathbf{x})| \leq \sum_{j}\sum_{\ell \in \mathbb{Z}}|\mathcal{T}_{\ell}b_j(\mathbf{x})|.
\end{align*}
By~\eqref{eq:kernel_for_cutoff} and the fact that $\supp b_j \subseteq Q_j$ and $\int b_j(\mathbf{y})\,dw(\mathbf{y})=0$,  we have
\begin{equation}\label{eq:int_ell_j}\begin{split}
&\int_{\mathbb{R}^N \setminus \Omega^{*}}|\mathcal{T}_{\ell}b_j(\mathbf{x})|\,dw(\mathbf{x})
 =\int_{\mathbb{R}^N \setminus \Omega^{*}}\left|\int_{Q_j} K_{\ell}(\mathbf{x},\mathbf{y})b_{j}(\mathbf{y})\,dw(\mathbf{y})\right|\,dw(\mathbf{x})\\
&\leq \int_{\mathbb{R}^N \setminus \mathcal O(Q_j^{*})}\left|\int_{Q_j}
 \Big(K_{\ell}(\mathbf{x},\mathbf{y})-K_\ell (\mathbf x,c_{Q_j})\Big)b_{j}(\mathbf{y})\,dw(\mathbf{y})\right|\,dw(\mathbf{x})\\
& \leq C M \min \Big((2^{\ell} \text{\rm diam} (Q_j))^{-\delta} , 2^\ell \text{\rm diam}(Q_j)\Big)\| b_j\|_{L^1(dw)}, 
\end{split}\end{equation}
where  in the last inequality we have used Proposition \ref{propo:integral}. Summing the inequalities~\eqref{eq:int_ell_j} over $j$ and $\ell$ we end up with 
$$\int_{\mathbb R^N\setminus \Omega^*} | \mathcal T_mb(\mathbf x)|dw(\mathbf x)\leq CM \sum_{j} \| b_j\|_{L^1}\leq CM \| f\|_{L^1(dw)},  $$
which, by the Chebyshev inequality, completes the proof of weak type $(1,1)$ of the operator $\mathcal T_m$. 

The strong type $(p,p)$ of $\mathcal T_m$ follows from the Marcinkiewicz interpolation theorem and a duality argument. \end{proof}

\section{Proof of Theorem\texorpdfstring{ \ref{teo:main_teo}~\eqref{numitem:hardy}}{1.1 (C)}}
\label{sec:Hardy}

Hardy spaces $H^1_{\Delta}$ in the Dunkl setting were studied in~\cite{conjugate},~\cite{Dziubanski18}, and for product systems of roots in~\cite{ABDH}. They are extensions of the classical Hardy spaces on  $\mathbb R^N$ introduced and developed in~\cite{SW},~\cite{FS}  (see also~\cite{St2}). 

We start this section by presenting three equivalent characterizations of the Hardy space $H^1_{\Delta}$ associated with the Dunkl theory. Then we shall prove Theorem~\ref{teo:main_teo}~\eqref{numitem:hardy}.

\begin{defn}\label{def:C-W-atoms} \normalfont
A function $a(\mathbf x)$ is an atom ($(1,\infty)$-atom) if there is a Euclidean ball $B$ such that
\begin{enumerate}[(A)]
 \item{$\supp a\subseteq B$;}\label{numitem:support}

 \item{$\| a\|_{L^\infty}\leq w(B)^{-1}$;}\label{numitem:size}

 \item{$\int a(\mathbf x)\, dw(\mathbf x)=0$.}\label{numitem:cancellations}
 \end{enumerate}
 \end{defn}

\begin{defn}\label{def:space_C-W} \normalfont A function $f$ belongs to $H^1_{\rm{atom}}$ if there are $\lambda_j\in \mathbb C$ and $(1,\infty)$-atoms $a_j$ such that $f=\sum_{j=1}^{\infty} \lambda_ja_j$ and $\sum_{j=1}^{\infty}|\lambda_j|<\infty$. Then
$$ \| f\|_{H^1_{\rm{atom}}} =\inf \Big\{ \sum_{j=1}^{\infty} |\lambda_j|\Big\}, $$
where the infimum is taken over all representations of $f$ as above.
\end{defn}

\begin{defn} \normalfont
We say that a function $f$ belongs to the real Hardy space $H^1_{\Delta}$  if the nontangential maximal function
\begin{align*}
 \mathcal M f(\mathbf x)=\sup_{\| \mathbf x-\mathbf y\|<t} |\exp(t^2\Delta )f(\mathbf x)|
\end{align*}
belongs to $L^1(dw)$. The space $H^1_{\Delta}$ is a Banach space with the norm
\begin{align*}
\| f\|_{H^1_{{\rm max}, H}}=\| \mathcal M f\|_{L^1(dw)}.
\end{align*}
\end{defn}

The following theorem was proved in~\cite[Theorem 1.6]{Dziubanski18}.
\begin{teo}\label{teo:main1} The spaces $H^1_{\Delta}$ and $H^1_{\rm{atom}}$ coincide and the corresponding norms are equivalent, that is, there is a constant $C>0$ such that
\begin{align}\label{eq:main}
C^{-1} \| f\|_{H^1_{\rm{atom}}}\leq \| f\|_{H^1_{{\rm max}, H}}\leq C \| f\|_{H^1_{\rm{atom}}}.
\end{align}
\end{teo}

\begin{defn}\normalfont
The Riesz transforms are defined in the Dunkl setting by
$$
R_jf=T_j(-{\Delta})^{-1\slash 2}f \text{ for }1 \leq j \leq N.
$$
\end{defn}
The Riesz transforms are bounded operators on $L^p({dw})$, for every $1<p<\infty$ (see \cite{AS},~\cite{ThangaveluXu2}).  In the limit case $p=1$, they turn out to be bounded operators from $H^1_{\Delta}$ into $H^1_{\Delta}\subset L^1({dw})$. This leads to consider the space
$$
H^1_{\rm Riesz}=\{f\in L^1({dw})\,{|}\,\| R_j f\|_{L^1(w)}<\infty\;{\forall\;1\le j\le N}\}.
$$
 The following theorem was proved in~\cite[Theorem 2.11]{conjugate}.
\begin{teo}\label{teo:main4}
The spaces $H^1_{\Delta}$ and $H^1_{\rm Riesz}$ coincide and the corresponding norms $\| f\|_{H^1_{{\rm max}, H}}$ and
$$
\|f\|_{H^1_{\rm Riesz}}:=\|f\|_{L^1({dw})}+\sum\nolimits_{j=1}^N\|R_jf\|_{L^1({dw})}.
$$
are equivalent.
\end{teo}

\begin{proof}[Proof of Theorem~\ref{teo:main_teo}~\eqref{numitem:hardy}]
Let us check first that there is $C>0$ such that $\|\mathcal{T}_ma\|_{L^1(dw)} \leq C$ for any atom $a(\cdot)$. Without loosing of generality we can assume that $a(\cdot)$ is associated with a cube $Q$. We have
\begin{align*}
\|\mathcal{T}_ma\|_{L^1(dw)} \leq \|\mathcal{T}_ma\|_{L^1(\mathcal{O}(Q^{*}),dw)}+\|\mathcal{T}_ma\|_{L^1((\mathcal{O}(Q^{*}))^{c},dw)},
\end{align*}
where $Q^{*}$ is the cube with the same center as $Q$ and two times larger side-length. By the Cauchy--Schwarz inequality and property~\eqref{numitem:size} of atom $a(\cdot)$ we have
\begin{align*}
\|\mathcal{T}_ma\|_{L^1(\mathcal{O}(Q^{*}),dw)} \leq w(\mathcal{O}(Q^{*}))^{1/2}\|\mathcal{T}_ma\|_{L^2(\mathcal{O}(Q^{*}),dw)} \leq C_1.
\end{align*}
Thanks to properties~\eqref{numitem:support} and~\eqref{numitem:cancellations} of $a(\cdot)$ we can use Proposition~\ref{propo:integral} and repeat argument presented in~\eqref{eq:int_ell_j} with $a(\cdot)$ instead of $b_j$. This leads us to $\|\mathcal{T}_ma\|_{L^1((\mathcal{O}(Q^{*}))^{c},dw)} \leq C_2$. 

We turn now to complete the proof of part~\eqref{numitem:hardy} of Theorem \ref{teo:main_teo}.  Since $\mathcal T_m$ maps continuously $L^1(dw)$ into $\mathcal S'(\mathbb R^N)$, it suffices
(by Theorems~\ref{teo:main1} and~\ref{teo:main4}) to check that there is a constant $C>0$ such that $\|R_j\mathcal{T}_{m}a\|_{L^1(dw)}\leq CM$ for any atom $a(\cdot)$ of $H^1_{\rm{atom}}$ and $j=1,2,\ldots, N$. For this purpose note that the operator $R_j\mathcal{T}_m$ is associated with the multiplier $n(\xi)=-i\frac{\xi_j}{\|\xi\|}m(\xi)$ and there is a constant $C_3>0$ such that
\begin{align*}
\sup_{t>0}\|\psi(\cdot)n(t\cdot)\|_{W^{s}_{2}} \leq C_3M.
\end{align*}
Therefore, we can repeat the argument presented above to the operator associated with $n$ in place of $\mathcal{T}_m$.
\end{proof}

\section{Case of \texorpdfstring{$L^1(dw)$}{L1(dw)} bounded translations}\label{sec:boundedL1}
\begin{teo}\label{teo:Thm3}  Assume that  for a root system $R$ and a multiplicity function $k\geq 0$ the translations $\tau_{\mathbf y}$ are uniformly bounded operators on $L^1(dw)$, that is, there is a constant $C>0$ such that for any $f\in L^1(dw)$ we have \begin{equation}\label{uniform_L1} 
\sup_{\mathbf y\in\mathbb R^N}\| \tau_{\mathbf y} f\|_{L^1(dw)} \leq C\| f\|_{L^1(dw)}. 
\end{equation}
 If $m$ is a bounded function on $\mathbb R^N$ such that ~\eqref{eq:assumption} is satisfied for  some $s>\mathbf N\slash 2$, then 
the multiplier operator $\mathcal T_m$ is  of weak-type $(1,1)$, bounded on $L^p(dw)$ for $1<p<\infty$, and bounded on the Hardy space $H^1_{\Delta}$. 
\end{teo}

\begin{remark}\normalfont  The same analysis applies to  any normalized system  of roots $R$, $k\geq 0$, and  radial multipliers, because the Dunkl transform of any radial function is radial and $\|\tau_{\mathbf y}f\|_{L^1(dw)}\leq \|  f\|_{L^1(dw)}$ for $f$ being radial. 
\end{remark}

\begin{remark}\normalfont
The inequality \eqref{uniform_L1} holds in the rank-one  case (see, e.g. \cite[Section 2.8]{Roesler-Voit} and hence in the product case.  
\end{remark}
\begin{proof}[Proof of Theorem \ref{teo:Thm3}]
Since $L^2(B(0,r), dw)$ is dense in $L^1(B(0,r), dw)$, by Theorem~\ref{teo:support} we have
\begin{equation}\label{eq:suppL1}
\text{\rm supp}\, \tau_{\mathbf x}f(-\,\cdot)\subset \mathcal O(B(\mathbf x,r)) \ \ \text{\rm for}\  f\in L^1(dw), \ \text{\rm supp}\, f\subseteq B(0,r).
\end{equation}
From \eqref{eq:suppL1} we easily deduce that for any $\delta\geq 0$ there is $C>0$ such that
\begin{equation}
\| f(\mathbf x,\mathbf y)(1+d(\mathbf x,\mathbf y))^\delta\|_{L^1(dw(\mathbf x))} \leq C\| f(\mathbf x)(1+\| \mathbf x\|)^\delta\|_{L^1(dw(\mathbf x))}.
\end{equation}
Indeed, let $f_j=\Psi_j f$, where $\Psi_j$ are defined in \eqref{eq:reproduce}. Then using \eqref{eq:notation},    \eqref{uniform_L1}, and \eqref{eq:suppL1}  we have 
\begin{equation}\begin{split}\label{eq:weightL1}
\| f(\mathbf x,\mathbf y)(1+d(\mathbf x,\mathbf y))^\delta\|_{L^1(dw(\mathbf x))} 
&\leq \sum_{j=0}^\infty \| f_j(\mathbf x,\mathbf y)(1+d(\mathbf x,\mathbf y))^\delta\|_{L^1(dw(\mathbf x))} \\
&\leq C'\sum_{j=0}^\infty 2^{j\delta}\| f_j(\mathbf x,\mathbf y)\|_{L^1(dw(\mathbf x))}\\
&\leq CC' \sum_{j=0}^\infty 2^{j\delta}\| f_j(\mathbf x)\|_{L^1(dw(\mathbf x))}\\
&\leq C'' \| f(\mathbf x)(1+\| \mathbf x\|)^\delta\|_{L^1(dw(\mathbf x))}. 
\end{split}\end{equation}
Let $m_\ell$, $m_{\ell, 1}$, $\tilde K_t(\mathbf x,\mathbf y)$, $K_t(\mathbf x,\mathbf y)$ be defined by \eqref{eq:function_m_l}, \eqref{eq:just_K}, and \eqref{eq:tilde_K} respectively. Take any $0\leq \delta<s-\mathbf N\slash 2$. Then 
\begin{equation}
\| \mathcal F m_\ell(\mathbf x)(1+\| \mathbf x\|^2)^{\mathbf N\slash 2+\delta}\|_{L^2(dw(\mathbf x))}  
+
\| \mathcal F m_{\ell,1}(\mathbf x)(1+\| \mathbf x\|^2)^{\mathbf N\slash 2+\delta}\|_{L^2(dw(\mathbf x))} \leq C_\delta  M, 
\end{equation}
which, by the Cauchy-Schwartz inequality  and \eqref{eq:homo}, imply 
$$ \| \mathcal F m_\ell (\mathbf x)(1+\| \mathbf x\|)^\delta\|_{L^1(dw(\mathbf x)} + \| \mathcal F m_{\ell,1} (\mathbf x)(1+\| \mathbf x\|)^\delta\|_{L^1(dw(\mathbf x))}  \leq C_\delta M.$$ 
By \eqref{eq:weightL1} and \eqref{eq:KK}, the kernels $K_\ell(\mathbf x,\mathbf y)$ satisfy \eqref{eq:K_ell_delta}. Further, the H\"older regularity  \eqref{eq:Holder_int} hold for $K_\ell(\mathbf x,\mathbf y)$  as well (see   \eqref{eq:Holder_tilde} for the proof). Hence we easily deduce that conclusion of Proposition \ref{propo:integral} is valid.  Finally the weak-type $(1,1)$ estimate and the boundedness on $L^p(dw)$ of the multiplier operator $\mathcal T_m$ are obtained by the standard Calder\'on-Zygmund analysis presented in Section \ref{sec:section_proof}. The proof of boundedness of $\mathcal T_m$ on the Hardy space $H^1_{\rm atom}$ is the same as in Section \ref{sec:Hardy}. 
\end{proof}

\end{document}